\documentclass[11pt,titlepage]{article}
\usepackage{amsmath,amsthm,amssymb,comment}
\begin{document}
\title{Computing Super Matrix Invariants}
\author{Allan Berele\\ Department of Mathematics\\ Depaul University\\ Chicago, IL 60614}
\maketitle
\begin{abstract}In \cite{BReg} we generalized the first and second fundamental theorems of invariant
theory from the general linear group to the general linear Lie superalgebra.
In the current paper we generalize the computations of the the numerical invariants
(multiplicities and Poincar\'e series) to the superalgebra case.  The results involve either inner
products of symmetric functions in two sets of variables, or complex integrals.

\end{abstract}
\newtheorem{thm}{Theorem}[section]
\newtheorem{cor}[thm]{Corollary}
\newtheorem{lem}[thm]{Lemma}
\newtheorem{coranddef}[thm]{Corollary and Definition}
\newtheorem*{conj}{Conjecture}
\newtheorem*{thm*}{Theorem}
\theoremstyle{definition}
\newtheorem{defn}[thm]{Definition}
\newtheorem{nota}[thm]{Notation}
\newtheorem{exa}[thm]{Example}
\theoremstyle{remark}
\newtheorem{rem}[thm]{Remark}
\newcommand{\lm}{\lambda}
\newcommand{\Ps}{Poincar\'e series}
\newcommand{\mkl}{M_{k,\ell}}
\newcommand{\glmn}{\gamma_{\mu,\nu}^\lm}
\newcommand{\hook}{k^2+\ell^2;2k\ell}
\newcommand{\bml}{\bar{m}_\lm}
\newcommand{\xy}{(1+x_iy_j^{-1})^{-1}(1+x_iy^{-1}_j)^{-1}}
\section*{Introduction}
Formanek's paper \cite{F}, based partly on Procesi's~\cite{P}, has been a major influence on my work.  One of the things Formanek does is describe a certain question which can be posed from five different points of view:  (1) Inner products of symmetric group characters; (2) plethysms of symmetric functions; (3) invariants of matrices; (4) trace rings of generic matrices; and (5) complex integrals.  Here is a brief description of each:
\begin{description}
\item [Problem 1] Let $\Lambda_k(n)$ be the partitions of~$n$ into at most~$k$ parts, let $\Lambda_k=\cup_n \Lambda_k(n)$, and let $\chi^\lambda$ denote the character of the symmetric group $S_n$ on the partition~$\lambda$.  Then the sum $\sum_{\lm\in\Lambda_k(n)}\chi^\lm\otimes\chi^\lm$ decomposes into irreducible characters as $\sum m_\lm\chi^\lm$ and we would like to evaluate the multiplicities $m_\lm$.
\item [Problem 2] Let $S_\lm$ denote the Schur function on the partition~$\lm$.  Let $X$ denote the set of variables $\{x_1,\ldots,x_k\}$, and let $S_\lm(XX^{-1})$ denote $S_\lm$ evaluated on all $xy^{-1}$, $x,y\in X$ (including $n$ 1's).  Then $m_\lm$ also equals $\langle S_\lm(XX^{-1}),1\rangle$, where the inner product is the natural inner product on symmetric functions as in~I.4 of~\cite{mac}.
\item [Problem 3] Consider functions $\phi: M_k(F)^n\rightarrow F$ which are polynomial in the entries and such that $$\phi(gA_1g^{-1},\ldots,gA_ng^{-1})=\phi(A_1,\ldots,A_n)$$ for all $A_1,\ldots,A_n\in M_k(F)$ and all $g\in GL_k(F)$.  Such a function is said to be invariant under conjugation from~$GL_k(F)$.  They form a ring with an $n$-fold grading and determine a Poincar\'e series $P(k,n)$.  Since $P(k,n)$ is a symmetric function it can be expanded into Schur functions as $$P(k,n)=\sum m_\lm S_\lm(t_1,\ldots,t_n)$$ where the $m_\lm$ are the heros of Problems 1 and 2.
\item [Problem 4] Let $X_\alpha$ be the generic $k\times k$ matrix with entries $X_\alpha =(x_{ij}^{(\alpha)})$, and let $R(k,n)$ be the algebra generated by $X_1,\ldots,X_n$.  Let $\bar{C}(k,n)$ be the commutative algebra generated by traces of elements of $R(k,n)$.  Then $\bar{C}(k,n)$ has an $n$-fold grading and so a \Ps\ in $n$ variables.  This series equals $P(k,n)$ from Problem~3.
\item [Problem 5] The function $P(k,n)$ can be evaluated as the following complex integral:  $$(2\pi i)^{-k}(k!)^{-1}\oint_T \frac{\prod_{i\ne j}(1-{z_i}{z_j}^{-1})}{\prod_{i,j} \prod_\alpha (1-z_i z_j^{-1} t_\alpha)}\frac{dz_1}{z_1}\wedge\cdots\wedge\frac{dz_k}{z_k}$$
where $i,j=1,\ldots,k$ and $\alpha=1,\ldots,n$, and where $T$ is the torus $|z_i|=1$.
\end{description}
Having five ways to look at the same object is useful for proving things about it.  Properties 3 and~5 are useful for actual computations, but also have theoretical consequences.  Using ~(3) it is easy to show that $m_\lm=0$ if $\lm\notin \Lambda_{k^2}$, and if $\lm\in\Lambda_{k^2}$ and $\mu=(\lm_1+a,\ldots,\lm_{k^2}+a)$, then $m_\lm=m_\mu$.  Using~(5) one can show that $P(n,k)$ is a rational function, that it can be written with denominator a product of terms of the form $(1-u)$, where $u$ is a monic monomial of degree at most~$k$, and that $P(n,k)$ satisfies the functional equation
\begin{equation} P(t_1^{-1},\ldots,t_n^{-1})=(-1)^g (t_1\cdots t_n)^{k^2}P(t_1,\ldots,t_k),
\label{eq:1}\end{equation}
where $g=(n-1)k^2+1$.

This theory has a $\mathbb{Z}_2$-graded analogue which we sketch briefly, see~\cite{B88}.  In this theory there are analogues of (1), (3) and~(4), but not of (2) and~(5).  Here are the analogues:
\begin{description}
\item [Problem 1a] Let $H(k,\ell;n)$ be the set of partitions of~$n$ in which at most $k$ parts are greater than or equal to~$\ell$, and let $H(k,\ell)=\cup_n H(k,\ell;n)$.  Using the standard notation for partitions $\lm=(\lm_1,\lm_2,\ldots)$ with $\lm_1\ge\lm_2\ge\cdots$, we have $$\lm\in H(k,\ell) \Longleftrightarrow \lm_{k+1}\le\ell.$$
The character of interest is $\sum\chi^\lm\otimes\chi^\lm$ summed over $\lm\in H(k,\ell;n)$ and we define the multiplicities $m_\lm$ by $$\sum_{\lm\in H(k,\ell;n)}\chi^\lm\otimes\chi^\lm=\sum m_\lm \chi^\lm.$$

\item [Problem 3a] Let $E$ be an infinite dimensional Grassmann algebra.  It has a natural $\mathbb{Z}_2$ grading.  We give the set $\{1,\ldots,k+\ell\}$ a $\mathbb{Z}_2$ grading via $$\deg(i)=\begin{cases}\bar{0}&\text{if }1\le i\le k\\ \bar{1}&\text{if }k+1\le i\le k+\ell
\end{cases}$$
and then grade the pairs $\{(i,j)\}_{i,j=1}^{k+\ell}$ via $\deg(i,j)=\deg(i)+\deg(j)$.

The algebra $\mkl$ is a subalgebra of $M_{k+\ell}(E)$ defined as the set of matrices $(a_{ij})$ such that for each $(i,j)$ the entry $a_{ij}\in E$ is homogeneous and has the same $\mathbb{Z}_2$ degree as $(i,j)$.
 Then $\mkl$ is an algebra and it has a non-degenerate trace with values in $E_0$ given by $$tr(a_{ij})=\sum (-1)^{\deg(i)}a_{ii}.$$  The group of units of $\mkl$  is denoted $PL(k,\ell)$ and is called the general linear Lie superalgebra.

Finally, we consider functions $\phi:\mkl^n\rightarrow E$, polynomial in the entries and invariant under conjugation from $PL(k,\ell)$.  These form an $n$-fold graded algebra with \Ps\ $\sum m_\lm S_\lm(t_1,\ldots,t_n)$, the same $m_\lm$ as in problem 1a.

\item [Problem 4a] Let $x_{ij}^{(\alpha)}$ be commuting indeterminants and let $e_{ij}^{(\alpha)}$ be anticommuting indeterminants, so that the algebra $S=F[x_{ij}^{(\alpha)},e_{ij}^{(\alpha)}]$ will be a free supercommutative algebra.  The generic matrix $A_\alpha$ will be the $(k+\ell)\times(k+\ell)$ matrix with $(i,j)$ entry equal to $x_{ij}^{(\alpha)}$ or $e_{ij}^{(\alpha)}$, depending on whether $\deg(i,j)$ equals $\bar{0}$ or $\bar{1}$, respectively.  Then the algebra $F[A_1,\ldots,A_n]$ will be the generic algebra for $\mkl$.  It has a trace function with image in~$S$, and we let $\bar{C}(k,\ell;n)$ be the algebra generated by the image of the trace map.  This ring has an $n$-fold grading and a \Ps\ in $n$ variables, the same series $\sum m_\lm S_\lm$ from 3a.
\end{description}
It is useful to push this last construction one step farther.  Let $B_\alpha$ be the $(k+\ell)\times(k+\ell)$ matrix with $(i,j)$ equal to $e_{ij}^{(\alpha)}$ if $\deg(i,j)$ is $\bar{0}$ and $x_{ij}^{(\alpha)}$ if  $\deg(i,j)=\bar{1}$, the opposite of the definition of $A_\alpha$.  Let $R(k,\ell;n,m)$ be the algebra generated by $A_1,\ldots,A_n$ and $B_1,\ldots,B_m$.  (For the reader familiar with the theory of magnums from~\cite{B85}, this is the magnum of $\mkl$.)  It has a supertrace function to $S$ and we let $\bar{C}(k,\ell;n,m)$ be the algebra generated by the traces.  $\bar{C}(k,\ell;n,m)$ has a $(k+\ell)\times(k+\ell)$ fold grading and a \Ps\ which can be expressed in terms of hook Schur functions as $$P(k,\ell;n,m)=\sum m_\lm HS_\lm(t_1,\ldots,t_n;u_1,\ldots,u_m),$$ where the $m_\lm$ are as in~1a.  See~\cite{BReg} for the theory of hook Schur functions.

Such a construction would also be possible in the matrix case (see~\cite{Bpre}), but it would be less useful.  A basic property of Schur functions is that $S_\lambda(t_1,\ldots,t_k)$ is non-zero precisely when $\lm\in\Lambda_k$.  And, in the matrix case, $m_\lm=0$ if $\lm\notin\Lambda_{k^2}$.  This means that we can reconstruct all the $m_\lm$ from the \Ps\ $P(k,n)$ as long as $n\ge k^2$.  In the  case of $M_{k,\ell}$, it is known that $HS_\lm(x_1,\ldots,x_n;y_1,\ldots,y_m)$ is non-zero if and only if $\lm\in H(n,m)$ and that $m_\lm\ne0$ only if $\lm\in H(k^2+\ell^2;2k\ell)$.  It follows that we get full information about the non-zero $m_\lm$ if we know the \Ps\ $H(k,\ell;n,m)$ for
 some $n\ge k^2+\ell^2$, $m\ge 2k\ell$.

 Our main goal in this paper is to present partial generalizations of problems~2 and~5 to the graded case.  Let $X$ denote the set of $k$ variables $\{x_1,\ldots,x_k\}$ and let $Y$ denote the set of $\ell$ variables $\{y_1,\ldots,y_\ell\}$.Then for certain $\lm$ which we call ``large'' and which include most of $H(k^2+\ell^2;2k\ell)$ we prove
 \begin{equation}
 m_\lm=\left\langle \prod_{ij}\xy HS_\lm(XX^{-1},YY^{-1};XY^{-1},YX^{-1}),1\right\rangle \label{eq:2}
 \end{equation}
 where the inner product is the inner product on functions symmetric on two sets of variables.  We will define it explicitly in section~1.

 Equation \eqref{eq:2} has an application in the case of typical $\lm$.  A partition in $H(a;b)$ but not in any strictly smaller hook is called typical, and the set of such is denoted $H'(a;b)$.  Such a partition can be thought of as being made up of three parts:  The $a\times b$ rectangle, a partition $\alpha(\lm)\in\Lambda_a$ to the right of the rectangle, and a partition $\beta(\lm)\in\Lambda_b$ whose conjugate lies below the rectangle.  Hopefully Figure~1 makes this clear, but if not we add that if $\lm=(\lm_1,\lm_2,\ldots)$ then $\alpha(\lm)=(\lm_1-b,\ldots,\lm_a-b)$ and $\beta(\lm)$ is the conjugate of $(\lm_{a+1},\lm_{a+2},\ldots)$.  The importance of typical partitions for our purposes lies in this factorization theorem for hook Schur functions from~\cite{BReg}.  Note that the number of  $x$'s and $y$'s in the theorem equal the dimensions of the hook.

 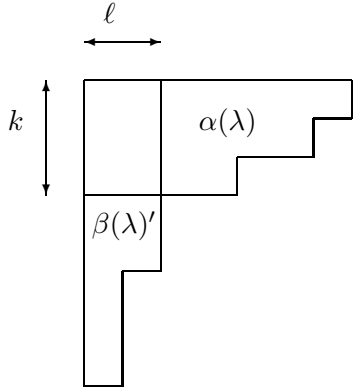
\begin{figure}\label{fig:1}
\setlength{\unitlength}{.2in}
\begin{center}
\begin{picture}(10,10)
\put(3,0){\line(0,1){8}}\put(4,0){\line(0,1){3}}\put(5,3){\line(0,1){5}}
\put(7,5){\line(0,1){1}}\put(9,6){\line(0,1){1}}\put(10,7){\line(0,1){1}}
\put(3,0){\line(1,0){1}}\put(3,5){\line(1,0){4}}\put(7,6){\line(1,0){2}}
\put(9,7){\line(1,0){1}}\put(3,8){\line(1,0){7}}\put(4,3){\line(1,0){1}}
\put(2,6){\vector(0,-1){1}}\put(2,6){\vector(0,1){2}}
\put(3,9){\vector(1,0){2}} \put(4,9){\vector(-1,0){1}}
\put(1,6.7){$k$} \put(3.5,9.5){$\ell$} \put(3.2,4){$\beta(\lm)'$}
\put(6,6.7){$\alpha(\lm)$}

\end{picture}\end{center}\caption{Definition of $\alpha(\lm)$ and $\beta(\lm)$}
\end{figure}

 \begin{thm}[The Factorization Theorem] If $\lm\in H'(a,b)$ with $\alpha=\alpha(\lm)$ and $\beta=\beta(\lm)$ then
 $$HS_\lm(x_1,\ldots,x_a;y_1,\ldots,y_b)=\prod(x_i+y_j) S_\alpha(x_1,\ldots,x_a) S_\beta(y_1,\ldots,y_b)$$
 \label{th:1}
 \end{thm}
 Combining the factorization theorem with \eqref{eq:2} we get the following.
 \begin{thm} Given $\lm,\mu\in H'(k^2+\ell^2;2k\ell)$ with $\alpha(\lm)=\alpha(\mu)+(a^{k^2+\ell^2})$ for some~$a$ and $\beta(\lm)=\beta(\mu)+(b^{2k\ell})$ for some $b$, then $m_\lm=m_\mu$.  (See Figure 2).\label{th:0.2}
 \end{thm}
 \begin{proof} $S_\alpha(\lm)(XX^{-1}, YY^{-1}) =$
 \begin{align*}
 (\prod_{z\in XX^{-1},YY^{-1}}z)^a  S_{\alpha(\mu)(XX^{-1}, YY^{-1}}) &\\
 1^a\cdot S_{\alpha(\mu)(XX^{-1}, YY^{-1}})&.
 \end{align*}
 And by the same token $S_{\beta(\lm)}(XY^{-1},YX^{-1})= S_{\beta(\mu)}(XY^{-1},YX^{-1})$.
 \end{proof}
  \begin{figure}\label{fig:2}
\setlength{\unitlength}{.2in}
\begin{center}
\begin{picture}(20,8)
\put(2,7){\line(1,0){7}}\put(12,7){\line(1,0){8}}\put(2,5){\line(1,0){4}}
\put(12,5){\line(1,0){5}}\put(6,6){\line(1,0){3}}\put(17,6){\line(1,0){3}}
\put(12,3){\line(1,0){2}}\put(3,3){\line(1,0){1}}\put(2,2){\line(1,0){1}}
\put(12,0){\line(1,0){1}}\put(13,1){\line(1,0){1}}
\put(2,2){\line(0,1){5}} \put(3,2){\line(0,1){1}}\put(4,3){\line(0,1){4}}
\put(6,5){\line(0,1){1}}\put(9,6){\line(0,1){1}}\put(12,0){\line(0,1){7}}
\put(13,0){\line(0,1){1}}\put(14,1){\line(0,1){6}}\put(17,5){\line(0,1){1}}
\put(20,6){\line(0,1){1}}\put(15,5){\line(0,1){2}}
\put(0,5){$\mu=$} \put(10,5){$\lm=$}
\put(14,7.5){\vector(1,0){1}} \put(15,7.5){\vector(-1,0){1}}
\put(14.5,4){\vector(0,1){1}}\put(14.5,4){\vector(0,-1){1}}
\put(14.5,8){$a$} \put(15,4){$b$}
\end{picture}\end{center}\caption{Theorem \ref{th:0.2}}
\end{figure}
 Turning to \Ps, since we do not know all of the $m_\lm$ we cannot hope to capture the full $\sum m_\lm S_\lm$ or $\sum m_\lm HS_\lm$, even in small numbers of variables.  There are two related infinite series we can express as integrals and derive information about.  For the first, let $m_\lm'$ be the right hand side of~(\ref{eq:2}), so $m_\lm=m_\lm'$ for large $\lm$.  Then we may define $$P'(k,\ell;a,b)=\sum m_\lm' HS_\lm(t_1,\ldots,t_a;u_1,\ldots,u_b)$$
 summed over all $\lm\in H(a,b)$.  Or instead, we may restrict  to typical partitions and define
 $$T(k,\ell;a,b)=\sum_{\lm\text{ typical}} m_\lm S_{\alpha(\lm)}(t_1,\ldots,t_a)S_{\beta(\lm)}(y_1,\ldots,y_b).$$
 Using (\ref{eq:2}) we can write each of these series as a complex integral over a torus.  The integrals are too complex to embellish an introduction, and perhaps too complex for much actual computation.  However, at least in the case of $T$ we can use the integral to prove that the
 $T(k,\ell;a,b)$ is the Taylor series of a rational function, we can
 describe what type of terms occur in the denominator, and we can prove a functional equation similar to~(\ref{eq:1}).

 Ideally we would like information about all of the $m_\lm$, not just for large or typical $\lm$.  If
 $$P(k,\ell;a,b)=\sum m_\lm HS_\lm(t_1,\ldots,t_a;u_1,\ldots,u_b),$$
 then it is still open whether $P(k,\ell;a,b)$ is a rational function, what its denominator looks like if it is, and
whether it satisfies a functional equation along the lines
of~\eqref{eq:1}.  See Corollary~\ref{cor:4.3} for a case in which it does not.

In  the  classical case described by Formanek one is also interested
in the character $$(\sum_{\lm\in \Lambda_k(n+1)} \chi^\lm\otimes
\chi^\lm)\downarrow =\sum \bar{m}_\lm \chi^\lm,$$ where the arrow indicates inducing down from $S_{n+1}$ to $S_n$.  Each of problems~1 through~5 have analogues in
this case.  The analogue of the invariant theory problem would concern the invariant maps $M_k(F)^n\rightarrow M_k(F)$. Generalizing to the $\mathbb{Z}_2$-graded case we would define  $$(\sum_{\lm\in H(k,\ell;n+1)} \chi^\lm\otimes
\chi^\lm)\downarrow =\sum \bar{m}_\lm \chi^\lm$$ and study the $\bar{m}_\lm$.  It turns out that analogues of problems~3 and~4 are known in this case, see~\cite{B88}, and that we can now develop analogues of problems~2 and~5, just like we did for~$m_\lm$.  The theory is very similar.  The formula for $\bar{m}_\lm$ for large enough~$\lm$ is the same as equation~\ref{eq:2} with an extra factor of $\sum x_i x_j^{-1}+\sum y_iy_j^{-1}$ and the same factor multiplies the integrands in the formulas for the power series $\bar{T}(k,\ell;a,b)$ and $\bar{P}'(k,\ell;a,b)$.
\section{Computation of Multiplicities}
\begin{nota} Given partitions $\mu,\nu\vdash n$ of the same~$n$, we define the coefficients $\glmn$ via the equation $\chi^\mu\otimes\chi^\nu=\sum_{\lm\vdash n} \glmn \chi^\lm$.  Note that from Problem~1a this implies \begin{equation}m_\lm=\sum\{\gamma_{\mu,\mu}^\lm\vert \mu\in H(k,\ell)\}.\label{eq:1.05}\end{equation}\label{not:1.1}
\end{nota}
We set the stage for the computation of $m_\lm$ by quoting two theorems
\begin{thm}[Berele-Regev~\cite{BReg}] If $\mu\in H(k_1,\ell_1)$ and $\nu\in H(k_2,\ell_2)$, then $\glmn=0$ unless $\lm\in H(k_1k_2+\ell_1\ell_2, k_1\ell_2+\ell_1 k_2)$\label{th:1.2}
\end{thm}
\begin{defn} Given $k,\ell$, we say that a partition $\lm$ is large if $\lm\in H(\hook)$ but $\lm\notin H(a^2+b^2;2ab)$ for any $a\le k$, $b\le\ell$ and at least one  of the inequalities strict.
Note that  typical partitions are all large.
\end{defn}
Using this definition we get this  corollary of Theorem~\ref{th:1.2}
\begin{lem} If $\lm$ is large and $\mu,\nu\in H(k,\ell)$, then $\glmn\ne0$ only if $\mu$ and $\nu$ are typical.\label{lem:1.2}
\end{lem}
The next theorem we need is due to Rosas from~\cite{Ro}.  It generalizes the classical result that given two sets of variables $X$ and $Y$, $S_\lm(XY)=\sum\glmn S_\mu(X)S_\nu(Y)$.  This paper of Rosas was our inspiration for our approach to $\glmn$.
\begin{thm}[Rosas] Given four sets of variables $X$, $Y$, $T$, and $U$,
$$HS_\lm( XT,YU;XU,YT) =\sum\glmn HS_\mu(X;Y)HS_\nu(T;U).$$\label{th:1.5}
\end{thm}
\begin{lem} Let $\lm\in H(\hook)$ be large and let $X$, $Y$, $T$, and $U$
be sets of variables with cardinalities $|X|=|T|=k$ and $|Y|=|U|=\ell$.  Then
$$\prod_{{x\in X}\atop{y\in Y}}(x+y)^{-1}\prod_{{t\in T}\atop{u\in U}}(t+u)^{-1} HS_\lm(XT,YU;XU,YT)$$
is a polynomial, symmetric in each of the four sets of variables.  Hence, it can be expanded in terms of Schur
functions.  This expansion involves only typical $\mu$ and $\nu$ and equals is
$$\sum \glmn S_{\alpha(\mu)}
(X)S_{\beta(\mu)}(Y)S_{\alpha(\nu)}(T)S_{\beta(\nu)}(U).$$
\label{lem:1.6}
\end{lem}
\begin{proof} By Theorem \ref{th:1.5} $$HS_\lm( XT,YU;XU,YT) =\sum\glmn HS_\mu(X;Y)HS_\nu(T;U).$$
By Lemma~\ref{lem:1.2} for each non-zero $\glmn$ $\mu$ and $\nu$ are typical and so we may apply the
Factorization Theorem, Theorem~\ref{th:1}
$$HS_\mu(X;Y)=\prod(x+y)S_{\alpha(\mu)}(X)S_{\beta(\mu)}(Y)$$
and $$HS_\nu(T;U)=\prod(t+u)S_{\alpha(\nu)}(T)S_{\beta(\nu)}(U).$$
The theorem now follows.
\end{proof}
\begin{defn} Given two finite sets of variables, the space
of polynomials $f(X,Y)$ which are symmetric in each has an
inner product with respect to which the $S_\mu(X)S_\nu(Y)$ are orthonormal.
If $p(X)$ is defined to be the product of the elements of~$X$ and $p(Y)$ to be
the product of the elements of $Y$, then the inner product extends to symmetric rational functions with denominator a power of $p(X)$ times a power of $p(Y)$
 such that $$\langle p(X)f(S,Y),p(X)g(X,Y)\rangle=\langle
f(X,Y),g(X,Y)\rangle$$\mbox{and }$$\langle p(Y)f(X,Y),p(Y)g(X,Y)\rangle=\langle
f(X,Y),g(X,Y)\rangle.$$ The inner product satisfies $\langle
f(X,Y),g(X,Y)\rangle$ equals the coefficient of~1 in
$(\Delta X)(\Delta Y)f(X,Y)g(X^{-1},Y^{-1})$, where $\Delta
X$ is the product of all $(1-\frac{x_1}{x_2})$ over all distinct pairs of elements of~$X$, and likewise for $\Delta
Y$.  This may also be expressed as a complex integral.  Say $|X|=n$ and $|Y|=m$.  Then $\langle f,g\rangle=$
$$\frac1{n!m!(2\pi i)^{n+m}}\oint_T (\Delta X)(\Delta Y)
f(X,Y)g(X^{-1},Y^{-1})\frac{dx_1}{x_1}\wedge\cdots
\wedge\frac{dy_m}{y_m}$$
where $T$ is the torus $|x_i|=|y_j|=1$.\label{def:1.7} Using this integral we may
speak of the inner product of any two functions of $X$ and $Y$.
\end{defn}
Here is our main theorem.
\begin{thm}\label{th:1.8}If $\lm\in H(\hook)$ is large, then $m_\lm$ equals $$\left\langle \prod\xy
HS_\lm(XX^{-1},YY^{-1};XY^{-1},YX^{-1}),1\right\rangle$$
where $|X|=k$ and $|Y|=\ell$.
\end{thm}
\begin{proof} By Lemma~\ref{lem:1.6}, noting that $(x+y)(x^{-1}+y^{-1})$ equals
$(1+xy^{-1})(1+x^{-1}y)$, the inner product
in the theorem equals
\begin{align*}
\langle &\sum\glmn S_{\alpha(\mu)}
(X)S_{\beta(\mu)}(Y)S_{\alpha(\nu)}(X^{-1})S_{\beta(\nu)}(Y^{-1}),1\rangle \\
&=\sum\glmn\langle S_{\alpha(\mu)}
(X)S_{\beta(\mu)}(Y),S_{\alpha(\nu)}(X)S_{\beta(\nu)}(Y)\rangle
\end{align*}
The inner product is either~1 or~0, depending on whether
$\alpha(\mu)=\alpha(\nu)$ and $\beta(\mu)=\beta(\nu)$.
But this happens precisely when $\mu=\nu$ and so the sum is simply $\sum\gamma_{\mu,\mu}^\lm$. \end{proof}
\begin{nota} For $X=\{x_1,\ldots,x_k\}$ and $Y=\{y_1,\ldots,y_\ell\}$ we let $Z_0=XX^{-1}\cup YY^{-1}$ and $Z_1=XY^{-1}\cup YX^{-1}$.  In this notation Theorem~\ref{th:1.8} can be stated as $$m_\lm=
\left\langle \prod_{z_1\in Z_1}(1+z_1)^{-1}
HS_\lm(Z_0;Z_1),1\right\rangle.$$
\end{nota}
\begin{rem} In Corolary 21 of \cite{B05} we proved that $m_\lm\le\langle HS_\lm(Z_0;Z_1),1\rangle$ for all~$\lm$
\end{rem}
\section{Integrals and \Ps}
Recall that $T(k,\ell;a,b)$ is the sum $\sum m_\lm S_{\alpha(\lm)}(A)S_{\beta(\lm)}(B)$, where the sum is over typical partitions and where $A$ and $B$ are sets of cardinality $a$ and $b$, respectively.  In order to compute this $T(k,\ell;a,b)$ from Theorem~\ref{th:1.8} we will need
Cauchy's identity, see \cite{mac}, I.4.3.  For any sets of variables $X$ and $Y$ Cauchy's identity states
$$\sum_\lm S_\lm(X) S_\lm(Y)=\prod_{x\in X \atop y\in Y}(1-xy)^{-1}.$$
We extend this slightly using  the factorization theorem,Theorem~\ref{th:1}.
\begin{align}\sum_{\lm\text{ typical}}& HS_\lm(A;B)S_{\alpha(\lm)}(C)S_{\beta(\lm)}(D)\notag \\&= \sum \prod(a+b) S_{\alpha(\lm)}(A)S_{\beta(\lm)}(B) S_{\alpha(\lm)}(C)S_{\beta(\lm)}(D)\notag\\& =\prod(a+b)\prod(1-ac)^{-1}\prod(1-bd)^{-1}\label{eq:3}\end{align}
where the $a$ runs over $A$, the $b$ runs over $B$, etc.
\begin{thm} $T(k,\ell;a,b)$ equals $(k!)^{-1}(\ell!)^{-1}(2\pi i)^{k+\ell}$ time the integral of
\begin{align*} & \prod(1+z_1)^{-1}\prod(z_0+z_1) \prod (1-az_0)^{-1} \\ &\prod (1-bz_1)^{-1} \prod_{i\ne j}(1-x_ix_j^{-1})\prod_{i\ne j}
(1-y_i y_j^{-1}) \frac{dx_1}{x_1}\wedge\cdots\wedge\frac{dy_b}{y_b}
\end{align*}
over the complex torus $|x_i|=1$,  $|y_i|=1$,  where the $a,b,z_0,z_1$ run over $A,B,Z_0$ and $Z_1$, respectively.\label{th:2.1}
\end{thm}
\begin{proof} By Theorem~\ref{th:1.8} if $\lm$ is typical
 $m_\lm S_{\alpha(\lm)}(A)S_{\beta(\lm)}(B)$ equals the inner product with~1 of
$$\prod(1+z_1)^{-1}
HS_\lm(Z_0;Z_1)S_{\alpha(\lm)}(A)S_{\beta(\lm)}(B).$$
Using equation~\eqref{eq:3} to sum this over all typical~$\lm$ we get the inner product with~1 of
$$\prod(1+z_1)^{-1}\prod(z_0+z_1) \prod (1-az_0)^{-1} \prod (1-bz_1)^{-1}.$$
Interpreting the inner product as an integral as described in Definition~\ref{def:1.7} completes the proof.
\end{proof}
In \cite{V} Van Den Bergh studied integrals over the torus of
the form $$f(z_1,\ldots,z_n)\prod_{i,j,k}(1-z_1z_j^{-1}x_k)^{-1}$$	 where $f$ is a degree~0 Laurent polynomial.  Since the product $\prod(1+x_iy_j^{-1})(1+x_iy^{-1}_j)$ divides evenly into $\prod(z_0+z_1)$ this is the
case here.  Using his results we get this corollary.
\begin{cor} $T(k,\ell;a,b)$ is a rational function.  The denominator can be written as a product of terms of
the form $(1-m)$ where $m$ is a monic monomial of degree at
most $k+\ell$, and of even degree in $U$.  If $a,b$ are so large that the integral
in Theorem~\ref{th:2.1} has no poles at~0, then $T(k,\ell;a,b)$
satisfies the functional equation
\begin{multline*}T(t_1^{-1},\ldots,t_a^{-1};u_1^{-1},\ldots,u_b^{-1})
=\\ (-1)^{(a-1)(k+\ell)+1}(t_1\cdots t_a)^{k^2+\ell^2}(u_1\cdots u_b)^{2k\ell}T(t_1,\ldots,
u_b)\end{multline*}
\label{cor:2.2}
\end{cor}

Finally, recall $P'(k,\ell;a,b)=\sum m_\lm' HS_\lm(A;B)$, where $m_\lm'=m_\lm$ for large~$\lm$.  The computation
of this function as an integral is similar to the computation
we just did for $T$, except that we need the following
theorem of Berele and Remmel instead of equation~\eqref{eq:3}.  We leave the proof of Theorem~\ref{th:2.4} to the
reader.
\begin{thm}(Berele-Remmel \cite{BRem}) Given four sets of variables
$A$, $B$, $C$ and $D$,
$$\sum_\lm HS_\lm (A;B)HS_\lm(C;D)=
\prod(1+ad)\prod(1+bc)\prod(1-ac)^{-1}\prod(1-bd)^{-1}.$$
\end{thm}
\begin{thm}\label{th:2.4} $P'(k,\ell;a,b)$ equals $(k!)^{-1}(\ell!)^{-1}(2\pi i)^{k+\ell}$ time the integral of
\begin{align*} & \prod(1+z_1)^{-1}
\prod(1+z_0b)\prod(1+z_1a)\prod(1-z_0a)^{-1} \\ &
\prod(1-z_1b)^{-1} \prod_{i\ne j}(1-x_ix_j^{-1})\prod_{i\ne j}
(1-y_i y_j^{-1}) \frac{dx_1}{x_1}\wedge\cdots\wedge\frac{dy_b}{y_b}
\end{align*}
\end{thm}
This integral is not of the type discussed by Van Den Bergh
and so we cannot easily derive an analogue of
Corollary~\ref{cor:2.2}.  It seems to us that $P'$ might be
less useful because of the potential presence of many terms $m_\lm'
HS_\lm$ with $m_\lm'$ not equal to $m_\lm$.
\section{Computation of $\bar{m}_\lm$}
As mentioned in the introduction, there is another set of multiplicities we are
interested in closely related to the~$m_\lm$.  Before showing how to compute
it for large or typical~$\lm$ we first describe the analogues of problems~1, 3~and 4.
\begin{description}
\item[Problem 1b] The sum $\sum(\chi^\lm\otimes\chi^\lm)\downarrow$ over $\lm\in
H(k,\ell;n+1)$ decomposes as a sum of irreducible characters $\sum\bml \chi^\lm$
and we would like to compute the $\bml$.
\item [Problem 3b] The functions $\phi:\mkl^n\rightarrow \mkl$ which are
polynomial in the entries and invariant under simultaneous conjugation from
$PL(k,\ell)$ form an $n$-graded ring with \Ps equal to $\sum \bml S_\lm(t_1,\ldots,t_n)$
\item[Problem 4b] Referring to the notation of Problem~4a, let $R(k,\ell;n)$
be the algebra generated by the generic matrices $A_1,\ldots,A_n$ together
with the trace ring $C(k,\ell;n)$.  then $R(k,\ell;n)$ is an $n$-graded ring
with \Ps\ $\sum \bml S_\lm(t_1,\ldots,t_n)$.  More generally, if $\bar{R}(k,\ell;n,m)$
is the algebra generated by $A_1,\ldots,A_n$, $B_1,\ldots,B_m$ and the supertrace ring $\bar{C}(k,\ell;n,m)$, then $\bar{R}(k,\ell;n,m)$ has \Ps\
$$\bar{P}(k,\ell;n,m)=\sum\bml HS_\lm(t_1,\ldots,t_n;u_1,\ldots,u_m).$$
\end{description}
We now turn to the computation of $\bml$.  Given an $S_n$-character $\chi=\sum\alpha_\lm \chi^\lm$, we define $H(\chi)$ to be $\sum \alpha_\lm HS_\lm(Z_0;Z_1)$.  This map respects addition and has two more properties we will need.  First, $H$ respects multiplication in the sense that \begin{equation} H(\chi_1\hat\otimes\chi_2)=H(\chi_1)H(\chi_2).\label{eq:4.1}\end{equation}  This follows from \cite{BReg}.  The second property of $H$ is a restatement of Theorem~\ref{th:1.8}:
\begin{lem} Let $\chi=\sum_{\lm\vdash n}\alpha_\lm \chi^\lm$ be such that $\alpha_\lm=0$ unless $\lm$ is large.  Then $$\langle \chi,\sum_{\mu\in H(k,\ell) }\chi^\mu\otimes \chi^\mu\rangle_{S_n}=
\left\langle \prod_{z_1\in Z_1}(1+z_1)^{-1}
H(\chi),1\right\rangle.$$
\end{lem}
\begin{proof} Each side of the equation equals $\sum \alpha_\lm m_\lm$.
\end{proof}
\begin{thm} For each large $\lm$ the multiplicity $\bml$ equals
$$\left\langle \sum_{z\in Z_0\cup Z_1} z\prod_{z_1\in Z_1}(1+z_1)^{-1}
HS_\lm(Z_0;Z_1),1\right\rangle.$$\label{th:3.2a}
\end{thm}
\begin{proof}  By definition, $\bml=\langle \chi^\lm, \sum(\chi^\mu\otimes\chi^\mu)\downarrow\rangle_{S_{n}}$, summed over $\mu\in H(k,\ell;n+1)$.  By Froebenius reciprocity this equals $\langle \chi^\lm\uparrow, \sum\chi^\mu\otimes\chi^\mu\rangle_{S_{n+1}}$ and $\chi^\lm\uparrow$ equals $\chi^{[1]}\hat\otimes\chi^\lm$.
Applying the previous lemma we get $$\bml=\left\langle \prod_{z_1\in Z_1}(1+z_1)^{-1}H(\chi^{[1]}\hat\otimes\chi^\lm),1\right\rangle.$$
By \eqref{eq:4.1}
\begin{align*} H(\chi^{[1]}\hat\otimes\chi^\lm)&=H(\chi^{[1]})H(\chi^\lm)\\
&=HS_{[1]}(Z_0;Z_1)HS_\lm(Z_0;Z_1)\\&= (\sum_{z\in Z_0\cup Z_1}z) HS_\lm(Z_0;Z_1).
\end{align*}
The theorem now follows.
\end{proof}
Theorem \ref{th:3.2a} easily implies analogues of Theorems \ref{th:2.1} and \ref{th:2.4}, and the former implies an analogue of Corollary~\ref{cor:2.2}.  The \Ps\ we study
are $\bar{T}(k,\ell;a,b)$ which equals the sum $\sum \bar{m}_\lm S_{\alpha(\lm)}(T)
S_{\beta(\lm)}(U)$ summed over typical~$\lm$; and $\bar{P}'(k,\ell;a,b)$ which equals the sum $\sum \bar{m}_\lm' HS_\lm(T;U)$, where $\bar{m}_\lm'$ is the inner product in Theorem~\ref{th:3.2a}.   Here are the results.
\begin{thm} $\bar{T}(k,\ell;a,b)$ equals $(k!)^{-1}(\ell!)^{-1}(2\pi i)^{k+\ell}$ time the integral of
\begin{align*} &\sum_{z\in Z_0\cup Z_1} z \prod(1+z_1)^{-1}\prod(z_0+z_1) \prod (1-az_0)^{-1} \\ &\prod (1-bz_1)^{-1} \prod_{i\ne j}(1-x_ix_j^{-1})\prod_{i\ne j}
(1-y_i y_j^{-1}) \frac{dx_1}{x_1}\wedge\cdots\wedge\frac{dy_b}{y_b}
\end{align*}
over the complex torus $|x_i|=1$,  $|y_i|=1$,  where the $a,b,z_0,z_1$ run over $A,B,Z_0$ and $Z_1$, respectively.\label{th:2.1a}
\end{thm}
\begin{cor} $\bar{T}(k,\ell;a,b)$ is a rational function.  The denominator can be written as a product of terms of
the form $(1-m)$ where $m$ is a monomial of degree at
most $k+\ell$, and of even degree in $U$.  If $a,b$ are so large that the integral
in Theorem~\ref{th:2.1a} has no poles at~0, then $T(k,\ell;a,b)$
satisfies the functional equation
\begin{multline*}T(t_1^{-1},\ldots,t_a^{-1};u_1^{-1},\ldots,u_b^{-1})
=\\ (-1)^{(a-1)(k+\ell)+1}(t_1\cdots t_a)^{k^2+\ell^2}(u_1\cdots u_b)^{2k\ell}T(t_1,\ldots,
u_b)\end{multline*}
\label{cor:2.2a}
\end{cor}
\begin{thm}\label{th:2.4a} $\bar{P}'(k,\ell;a,b)$ equals $(k!)^{-1}(\ell!)^{-1}(2\pi i)^{k+\ell}$ time the integral of
\begin{align*} & \sum_{z\in Z_0\cup Z_1} z\prod(1+z_1)^{-1}
\prod(1+z_0b)\prod(1+z_1a)\prod(1-z_0a)^{-1} \\ &
\prod(1-z_1b)^{-1} \prod_{i\ne j}(1-x_ix_j^{-1})\prod_{i\ne j}
(1-y_i y_j^{-1}) \frac{dx_1}{x_1}\wedge\cdots\wedge\frac{dy_b}{y_b}
\end{align*}
\end{thm}

\section{Examples}
\subsection{Case of $\lm$ one row or one column}
We first compute $m_\lm$ for $\lm=(n)$ and $\lm=(1^n)$ and $k,\ell$ arbitrary.  Such $\lm$ will be neither large nor typical and we compute it directly from the inner product definition $m_\lm=\sum \chi^\lm\otimes\chi^\lm$ summed over $\lm\in H(k,\ell;n)$.
\begin{lem}  With $\gamma$ as in Notation~\ref{not:1.1},  $\gamma^{(n)}_{\mu,\nu}=\delta_{\mu,\nu}$ and $\gamma^{(1^n)}_{\mu,\nu}=\delta_{\mu,\nu'}$\label{lem:3.1}
\end{lem}
\begin{proof} Referring to I.7 of \cite{mac}, $\chi^{(n)}\otimes\chi^\mu=\chi^\mu$ for any $\mu$ and so $\gamma^\mu_{(n),\nu}=\delta_{\mu,\nu}$.  But, considered as a function of $\lm,\mu,\nu$,  $\gamma^\lm_{\mu,\nu}$ is symmetric and so $\gamma^{(n)}_{\mu,\nu}=\delta_{\mu,\nu},$ as claimed.  The second statement follows similarly using the identity $\chi^{(1^n)}\otimes\chi^\mu=\chi^{\mu'}$ from example 2 in section I.7 of~\cite{mac}.
\end{proof}
\begin{thm}  $m_{(n)}$ equals the number of partitions in $H(k,\ell;n)$ and $m_{(1^n)}$ equals the number of self-conjugate partitions in $H(k,\ell;n)$.  In particular, $m_{(n)}\ne m_{(1^n)}$ and so $m_\lm\chi^\lm$ is not symmetric under conjugation.
\end{thm}
\begin{proof}
By definition $m_{(n)}=\sum \gamma^{(n)}_{\mu,\mu}$, summed over $\mu\in H(k,\ell;n)$; and by lemma~\ref{lem:3.1} each $\gamma^{(n)}_{\mu,\mu}$ equals~1 and so the sum is $|H(k,\ell;n)|$.  The case of $m_{(1^n)}$ is similar, with $\gamma^{(1^n)}_{\mu,\mu}$ equaling~1 if $\mu$ is self-conjugate and~0 otherwise.
\end{proof}

\begin{cor} Let $f(x)=P(1,1;0,1)$, $g(x)=P(2,2;1,0)$, and let $h(x)$ equal either
$f(x)$ or $g(x)$.  Then $h(x)$ does not satisfy a functional equation $h(x^{-1})=\pm x^a h(x)$.\label{cor:4.3}
\end{cor}
\begin{proof}A partition $\lm\in H(1,1)$ is self-conjugate if and only if $\lm =[0]$ or $\lm=[a+1,1^a]$ for some $a\ge 0$.  Hence, $$f(x)=1+\sum x^{2a+1} =1+\frac x{1-x^2}=\frac{1+x-x^2}{1-x^2}.$$

For $g(x)$, note that if $\lm\in H(2,2)$, either $\lm$ is typical or $\lm\in H(1,1)$.  If $\lm$ is typical, then $|\lm|=4+|\alpha(\lm)|+|\beta(\lm)|$, and
since $\alpha(\lm)$ and $\beta(\lm)$ are partitions of height at most~2, it follows
that $$\sum_{\lm\in H'(2,2)}t^{|\lm|}=\frac{t^4}{[(1-t)(1-t^2)]^2}.$$
If $\lm\in H(1,1)$,then either $\lm=[0]$ or $\lm$ is typical.  It follows that
$$\sum_{\lm\in H(1,1)}t^{|\lm|}=1+\frac x{(1-x)^2}.$$
Adding, we get
\begin{align*}
g(x)&=1+\frac x{(1-x)^2}+\frac {x^4}{(1-x)^2(1-x^2)^2}\\
         &=\frac{1-x+x^2+x^3-x^4}{(1-x)^2(1-x^2)^2}.
         \end{align*}
  It is now easy to see that neither $f(x)$ not $g(x)$ satisfy a functional equation
  of the indicated type.
  \end{proof}
\subsection{The case of $(k,\ell)=(1,1)$}
We now turn to the $k=\ell=1$ case.  This case was computed directly by Regev in~\cite{R} and  Remmel in~\cite{remmel}, so this computation is mostly a check on Theorem~\ref{th:2.1}.  In this case $X=\{x\}$ and $Y=\{y\}$, and so $Z_0=\{1,1\}$ and $Z_1=\{\frac xy, \frac yx \}$.  Hence, after canceling the $(1+\frac xy)(1+\frac yx)$ in the denominator with one  of the two in the numerator, the
integrand in Theorem~\ref{th:2.1} equals
\begin{equation}
 (1+\frac xy)(1+\frac yx)\prod_i(1-t_i)^{-2} \prod_j(1-\frac xy u_j)^{-1}(1-\frac yx u_j)^{-1}
\end{equation}
Since $(\hook)=(2,2)$, we will take the case of two $t$'s and two $u$'s.  The first integral
is by $\frac{dx}x$ and the poles inside the circle are at $x=yu_1$ and $x=yu_2$.  Pulling
out the factor of $(1-t_1)^2(1-t_2)^2$, the pole at $yu_1$ equals:
$$\frac{(1+u_1^{-1})(1+u_1)}{(1-u_1u_2)(1-u_2u_1^{-1})
(1-u_1^2)}$$
which simplifies to
$$\frac{(1+u_1)}{(1-u_1u_2)(u_1-u_2)(1-u_1)}.$$
Integrating by $\frac{dy}y$ leaves this as it is.  Similarly, if we use the pole at $x=yu_2$ we
simply switch $u_1$ and $u_2$ yielding
$$\frac{(1+u_2)}{(1-u_1u_2)(u_2-u_1)(1-u_2)}.$$
Adding these two fractions and multiplying back the $(1-t_1)^2(1-t_2)^2$ we factored out  we get $T(1,1;2,2)$.
\begin{thm}\label{th:3.3} $T(1,1;2,2)=$
$$\frac2{(1-t_1)^2(1-t_2)^2(1-u_1)(1-u_2)(1-u_1u_2)}.$$
\end{thm}
It follows from the identities $\sum S_\lambda (u_1,u_2)=(1-u_1)^{-1}(1-u_2)^{-1}(1-u_1u_2)^{-1}$ and $\sum (\lm_1-\lm_2+1)S_\lambda (t_1,t_2)=(1-t_1)^{-2}(1-t_2)^{-2}$ that if $\lm$ and typical with $\alpha(\lm)=(\alpha_1,\alpha_2)$ and $\beta(\lm)=(\beta_1,\beta_2)$, then $m_\lm=2(\alpha_1-\alpha_2+1)$, in agreement with~\cite{remmel}.

In order to find $m_\lm$ for smaller $\lm$, note that every $\lm\in H(1,1)$ is typical, except
for $\lm=[0]$, and so $m_\lm=m_\lm'$ for all $\lm$ except for $m_{[0]}=1$ and $m_{[0]}'=0$.
Hence, we may use $P'(1,1;1,1)$ to compute $m_\lm$ in the $(1,1)$-hook.  By Theorem~\ref{th:2.4}, we compute $P'(1,1;1,1)$ by integrating
\begin{equation}\frac{(1+u)^2(1+\frac xyt)(1+\frac yxt)}{(1+\frac xy)(1+\frac yx)(1-t)^2
(1-\frac xy u)(1-\frac yx u)}\frac{dx}x\wedge\frac{dy}y\label{eq:6}
\end{equation}
There are two poles:  One at $x=yu$ and one at $x=-y$.  The former has residue
$$\frac{(1+u)^2(1+ut)(1+\frac tu)}{(1+u)(1+u^{-1})(1-t)^2(1-u^2)}$$ which equals
$$\frac{(1+ut)((t+u)}{(1-t)^2(1-u^2)}$$
times $\frac{dy}y$.  The pole in \eqref{eq:6} at $x=-y$ is of order two, so in order to compute the residue
we must first multiply by $(x+y)^2$, then take the partial derivative with respect to~$x$, and
finally substitute $-y$ for~$x$.  The computation is a bit long and the result is~0.
\begin{thm} $P'(1,1;1,1)=(1+ut)(t+u)(1-t)^{-2}(1-u^2)^{-1}$.\end{thm}
It follows that if $\lm=(a+1,1^b)$, then $m_\lm$ equals $a+1$ if $b$ is even and $a$ if $b$ is odd.

A similar analysis can be carried out for $\bar{m}_\lm$.  We leave the proof to the reader.
\begin{thm} $\bar{P}'(1,1;1,1)=1+2(t+u)(1+tu)(1-t)^{-2}(1-u)^{-1}$ which implies that $\bar{m}_{[a+1,1^b]}$ equals $4a+2$ if $b>0$ and $\bar{m}_{[a+1]}=2a+2$.
Also,
\begin{align*}\bar{P}'(1,1;2,2)&=\frac{2(3+u_1+u_2-u_1u_2)}{(1-t_1)^2(1-t_2)^2(1-u_2)(1-u_2)(1-u_1u_2)}\\ &=\frac1{(1-t_1)^2(1-t_2)^2}\left(\frac8{(1-u_2)(1-u_2)(1-u_1u_2)}-\frac2{1-u_1u_2}
\right)
\end{align*}
It follows that if $\lm$ is typical with $\alpha(\lm)=(\alpha_1,\alpha_2)$ and $\beta(\lm)=(\beta_1,\beta_2)$, then $\bar{m}_\lm=8(\alpha_1-\alpha_2+1)$ unless $\beta_1=\beta_2$ in which case $\bar{m}_\lm=6(\alpha_1-\alpha_2+1)$.
\end{thm}
\subsection{The case of $(k,\ell)=(2,1)$}
We conclude with a peek into the unknown, thanks to a Maple computation:
\begin{equation}T(2,1;1,0)=\frac{372+801t+835t^2+515t^3+213t^4+35t^5+t^6}{(1-t)^2(1-t^2)},\end{equation}
\begin{multline}T(2,1;0,1)=(1-u)^{-1}(372+780u+1083u^2+1193u^3+1034u^4\\ +754u^5+513u^6+319u^7+158u^8+54u^9+11u^{10}+u^{11}),\end{multline}
\begin{equation} \bar{T}(2,1;1,0)=\frac{2697++6346t+6641t^2+4449t^3+1981t^4+503t^5+50t^6+t^7}{(1-t)^2(1-t^2)},\end{equation}
and
\begin{multline}\bar{T}(2,1;0,1)=(1-u)^{-1}(2697+6249u+8817u^2+9587u^3+8706u^4\\
+6890u^5+4877u^6+3107u^7+1744u^8+820u^9+301u^{10}\\ +79u^{11}+13u^{12}+u^{13})
\end{multline}

 \end{document}